\documentclass[12pt]{article}

\usepackage{amsmath,amsthm,amssymb}
\usepackage{enumitem}
\usepackage[hidelinks]{hyperref}
\usepackage{geometry}
\allowdisplaybreaks[4]

\theoremstyle{plain}
\newtheorem{theorem}{Theorem}
\newtheorem{lemma}[theorem]{Lemma}
\newtheorem{corollary}[theorem]{Corollary}
\newtheorem{claim}{Claim}

\newtheorem{proposition}[theorem]{Proposition}

\def \h{\mathcal{H}}
\def \G{\mathcal{G}}

\begin{document}

\title{A solution to Csikv\'ari's conjecture and the largest matching root of $k$-graphs \thanks{Supported by the National Natural Science Foundation of China (No. 12331012, 12571360)}}

\author{Jiang-Chao Wan \thanks{School of Mathematics and Statistics, Hefei University, Hefei 230601, Anhui, China},  Yi Wang \thanks{School of Mathematical Sciences, Anhui University, Hefei 230601, Anhui, China;   Corresponding author: wangy@ahu.edu.cn}}

\date{}

\maketitle

\begin{abstract}
In 2011, Csikv\'ari
[Electron. J. Combin. {\bf 18} (2011), $\#$P182]
proved that among all graphs with a prescribed number of edges, the largest matching root is attained by a threshold graph, and conjectured that the extremal graph should be `as star-like as possible.'
In this paper, we give a complete and affirmative answer to this problem and extend it to the setting of uniform hypergraphs.
We prove that for every $k$-graph $\mathcal{H}$ with $m$ edges, its largest matching root satisfies
$$\lambda(\mathcal{H})\le m^{1/k},$$
with equality if and only if $\mathcal{H}$ is intersecting.
For $k=2$, after deleting all isolated vertices, the resulting graph must be the star $K_{1,m}$ or a triangle, thereby confirming Csikv\'ari's conjecture.
Moreover, if the matching number $\nu(\mathcal{H})\ge 2$,
then
\[
\lambda(\mathcal{H})\le \left(\frac{m+\sqrt{m^2-4(\nu(\mathcal{H})-1)}}{2}\right)^{1/k},
\]
with equality if and only if $\nu(\mathcal{H})=2$ and $\mathcal{H}$ has exactly one $2$-matching.
\end{abstract}



\section{Introduction}

Throughout this paper, all graphs are finite, simple, and undirected.
Let $G$ be an $n$-vertex graph.
A \emph{matching} of $G$ is a subset of its edges such that no two share a common vertex.
Denote by $p(G,i)$ the number of matchings in $G$ consisting of $i$ edges.
In their seminal work, Heilmann and Lieb~\cite{Heilmann} introduced the \emph{matching polynomial} of graph $G=(V,E)$ as follows
\[
\mu(G,x)=\sum_{i=0}^{\lfloor n/2\rfloor}(-1)^i p(G,i)\,x^{n-2i}.
\]
The celebrated Heilmann--Lieb~\cite{Heilmann} theorem states that all zeros of $\mu(G,x)$ are real and symmetric about the origin.
Consequently, the largest real root $\lambda(G)$ of $\mu(G,x)$ is a well-defined real number, and we simply call it the \emph{largest matching root} of $G$.
The extremal problem concerning the largest matching root dates back to the celebrated paper by Lov\'asz
and Pelik\'an~\cite{LovaszPe} in 1973,
where they determined the first two largest matching roots among trees with a prescribed number of edges.
Thereafter, many researchers studied the problem of determining the graphs for which the largest matching root achieves the maximum in a given family of graphs (see, e.g.,~\cite{Csikvari,Liu,Suejc}).

Confirming a conjecture of Brualdi and Hoffman~\cite{Brualdi}, Rowlinson \cite{Rowlinson} proved that the graph $G_m$ attains the maximum spectral radius among all graphs with $m=\binom{s}{2}+t$ edges, where $G_m$ is obtained from the complete graph $K_s$ by adding a new vertex and $t$ new edges.
Inspired by this classical result in spectral graph theory,
in 2011, Csikv\'ari \cite{Csikvari} proved that among all graphs with a prescribed number of edges, the maximum matching root is always attained by a \emph{threshold graph}.
In the concluding remarks of his paper, he posed the following natural extremal problem:

\begin{quote}
\emph{``Now if we consider the problem of finding the graph
maximizing the largest root of the matching polynomial among graphs with prescribed number of edges, the situation is completely different.
We believe that the Kelmans transformation works because it generates some large-degree vertices.
We conjecture that in this case the extremal graph will be as `star-like' as it is possible:
it has as many vertices of degree $n-1$ as it is possible and one more vertex of the clique part of the threshold graph has some additional edges.''}
\begin{flushright}
--- Csikv\'ari \cite[Section~11]{Csikvari}
\end{flushright}
\end{quote}

In this paper, we give a solution to Csikv\'ari's conjecture.
We determine all graphs whose largest matching root attains the maximum among graphs with a prescribed number of edges.
The star $K_{1,m}$ with $m$ edges and the triangle are the extremal graphs.


\begin{theorem}\label{CsikvariConjgraph}
Let $G$ be a graph with $m$ edges, and let $\lambda(G)$ be the largest matching root of $G$.
Then
$$\lambda(G)\leq \sqrt{m}.$$
Moreover, equality holds if and only if, possibly with some isolated vertices added, $G$ is
\begin{itemize}
\item the star $K_{1,m}$ for $m=1,2$ or $m\ge 4$, and

\item either the star $K_{1,3}$ or the triangle for $m=3$.

\end{itemize}
%
\end{theorem}

Moreover, our methods allow us to solve Csikv\'ari's problem in the setting of uniform hypergraphs.


A {\it $k$-uniform hypergraph} (or {\it $k$-graph} for short) $\h=(V(\h), E(\h))$ consists of a vertex set $V(\h)$ and an edge set $E(\h)$,
where each edge $e\in E(\h)$ is a $k$-element subset of $V(\h)$.
Clearly, $2$-graph are ordinary graphs.
A {\em matching} in $\h$ is a set of vertex-disjoint edges, and an $r$-matching is a matching consisting of $r$ edges.
The maximum number of edges in a matching of $\h$ is called the {\em matching number} of $\h$ and denoted $\nu(\h)$.
To study the spectral radius of the adjacency tensor of $k$-graphs, Su et al.~\cite{Suejc} introduced the {\em matching polynomial} of a $k$-graph $\h$ as follows:
$$
\mu(\h,x)=\sum_{r\geq 0}(-1)^rp(\h,r) x^{|V(\h)|-kr},
$$
where $p(\h,r)$ denotes the number of $r$-matchings of $\h$.

Denote by $\lambda(\h)$ the maximum modulus of the zeros of the matching polynomial of a $k$-graph $\h$.
Recently, the authors~\cite{WanWF} established a hypergraph Heilmann--Lieb theorem.
In particular, they proved that for a connected $k$-graph $\h$, $\lambda(\h)$ is a simple root of $\mu(\h,x)$,
so $\lambda(\h)$ is the largest root of $\mu(\h,x)$
(see~\cite[Theorem 1.2]{WanWF} for details).
Therefore, it is reasonable to refer to $\lambda(\h)$ as the {\it largest matching root} of $\h$ and to consider extremal problems for $\lambda(\h)$.


The main result of this paper provides the following universal bound, valid for every uniformity $k\geq 2$.

\begin{theorem}\label{CsikvariConjHgraph}
Let $k\ge 2$ and let $\mathcal{H}$ be a $k$-graph with $m$ edges.
Then
$$
\lambda(\mathcal{H})\leq m^{1/k}.
$$
with equality if and only if $\mathcal{H}$ is intersecting, i.e.\ every two edges have a non-empty intersection.
Moreover, if $\nu(\mathcal H) \geq 2$, then
$$
\lambda(\mathcal H)\leq
\left(\frac{m+\sqrt{m^2-4(\nu(\mathcal H)-1)}}{2}\right)^{1\over k},
$$
with equality if and only if $\nu(\mathcal H)=2$ and $p(\mathcal H,2)=1$.
\end{theorem}

When $k=2$, the condition that every two edges intersect forces the graph to be a star $K_{1,m}$ for $m=1,2$ or $m\ge 4$, and yields two extremal graphs $K_{1,3}$ and $K_3$ when $m=3$.
Thus, Theorem~\ref{CsikvariConjHgraph} immediately implies Theorem~\ref{CsikvariConjgraph}.

For $k\geq 3$, the extremal family is richer:
there exist numerous intersecting $k$-graphs with $m$ edges that attain the extremal bound in Theorem~\ref{CsikvariConjHgraph}.
Furthermore, for every $k\ge 3$ and every $m\ge 2$, we can construct a $k$-graph $\mathcal H_m^{(k)}$ with $m$ edges such that $\nu(\mathcal H_m^{(k)})=2$ and $p(\mathcal H_m^{(k)},2)=1$ as follows:
Take two disjoint edges $e_1$ and $e_2$ with $u\in e_1$ and $v\in e_2$,
and define the remaining $m-2$ edges as $f_i = \{u,v\}\cup W_i$ for $i=1,\dots,m-2$,
where $W_1,\dots,W_{m-2}$ are pairwise disjoint $(k-2)$-sets of new vertices and each $W_i$ disjoint from $e_1\cup e_2$.



A $k$-graph $\mathcal H$ is called {\em $t$-intersecting} if $|A\cap B|\geq t$ for every $A,B\in E(\mathcal H)$.
Erd\H{o}s, Ko, and Rado~\cite{Erdos} proved that there exists an integer $n_0(k,t)$ such that if $\mathcal H$ is an $n$-vertex $t$-intersecting $k$-graph and $n\geq n_0(k,t)$, then
\[
|E(\mathcal H)| \leq \binom{n-t}{k-t}.
\]
Wilson~\cite{Wilson} showed that the smallest possible such $n_0(k,t)$ is $(t+1)(k-t+1)$,
and for $n> (t+1)(k-t+1)$, equality holds if and only if $\mathcal H$ is a $t$-star, i.e.,
$E(\mathcal H)=\{S\in \binom{[n]}{k}: T\subseteq S\}$ for a fixed $T\in \binom{[n]}{t}$.
Note that Theorem~\ref{CsikvariConjHgraph} implies $\lambda(\mathcal{H})=|E(\mathcal H)|^{1/k}$ for every intersecting $k$-graph $\mathcal H$.
Combining these results, we obtain the following largest matching root version of the Erd\H{o}s--Ko--Rado theorem.

\begin{corollary}
Let $k\geq 2$, $t\geq 1$,
and let $\mathcal{H}$ be an $n$-vertex $k$-graph.
If $\mathcal{H}$ is $t$-intersecting and $n\geq (t+1)(k-t+1)$, then
$$\lambda(\mathcal{H})\leq \binom{n-t}{k-t} ^{1/k},$$
with equality for $n>(t+1)(k-t+1)$ if and only if $\mathcal H$ is a $t$-star.
\end{corollary}

\section{Proof of Theorem~\ref{CsikvariConjHgraph}}

This section is devoted to proving Theorem~\ref{CsikvariConjHgraph}.
We begin with the following lemma, which will be used when $k$-graphs are disconnected.

\begin{lemma}[\cite{Suejc}]\label{disjointunion}
Let $\G$ and $\h$ be two vertex-disjoint $k$-graphs.
For the disjoint union $\G\oplus\h$ of $\G$ and $\h$, we have	
    $$\mu(\G\oplus \h, x)=\mu(\G, x)\mu(\h, x).$$
\end{lemma}

\begin{lemma}[\cite{WanWF}]\label{prelemma}
Let $\h$ be a connected $k$-graph.
Then $\lambda(\mathcal{H})$ is a simple root of $\mu(\h,  x)$
and $\mu(\h,  x)$ has exactly $k$ distinct zeros with modulus $\lambda(\h)$
    $$\lambda(\h) e^{\frac{2\pi j }{k}\mathbf{i}}, j = 0, 1,\ldots, k-1,$$
      where $\mathbf{i}$ is the imaginary unit.
\end{lemma}

To simplify dealing with $\lambda(\mathcal{H})$, we introduce the following polynomial:
$$\phi(\mathcal H,x)=\sum_{i=0}^{\nu(\mathcal{H})}(-1)^i p(\mathcal H,i)\,x^{\nu(\mathcal{H})-i}.$$
Observe that
\[
\mu(\mathcal H,x)
=x^{|V(\mathcal H)|-k\nu(\mathcal{H})}\phi(\mathcal H,x^k),
\]
and the non-zero zeros of $\mu(\mathcal H,x)$ are precisely the $k$-th roots of the zeros of $\phi(\mathcal H,x)$.
Denote by $\theta(\mathcal{H})$ the maximum modulus of the zeros of $\phi(\mathcal H,x)$.
If $\mathcal{H}$ has $t$ components $\mathcal{H}_1, \ldots, \mathcal{H}_t$, then Lemmas~\ref{disjointunion} and~\ref{prelemma} implies that
$$\lambda(\mathcal{H})^k
=\max_{1\leq i \leq t}  \{\lambda(\mathcal{H}_i)^k\}
=\max_{1\leq i \leq t}  \{\theta(\mathcal{H}_i)\}
 =\theta(\mathcal{H}).$$
So, for every $k$-graph $\mathcal{H}$, we have $\theta(\mathcal{H})=\lambda(\mathcal{H})^k$ and $\theta(\mathcal{H})$ is the largest real root of $\phi(\mathcal H,x)$.
Therefore, one may translate the extremal problems for $\lambda(\mathcal{H})$ into the corresponding problems for $\theta(\mathcal{H})$.

\begin{lemma}\label{lem:rec}
Let $\mathcal H$ be a $k$-graph with $m$ edges. For $1\leq i<\nu(\mathcal H)$,
\[
(i+1)p(\mathcal H, i+1)\leq (m-i)p(\mathcal H, i).\]
\end{lemma}
\begin{proof}
We double-count the number $N$ of ordered pairs $(M,e)$, where $M$ is an $(i+1)$-matching of $\mathcal H$ and $e$ is an edge contained in $M$.
On the one hand, for each $(i+1)$-matching $M$ in $\mathcal H$, there are $i+1$ choices of $e\in M$.
So we get $N=(i+1)p(\mathcal H,i+1)$.
On the other hand,
we may construct such a pair $(M,e)$
by first choosing an $i$-matching $M'$ in $\mathcal H$ and then an edge $e\notin M'$ such that $M = M'\cup\{e\}$ is an $(i+1)$-matching.
Clearly, for a fixed $i$-matching $M'$, the number of edges $e\notin M'$ is at most $m-i$.
So we deduce that $(i+1)p(\mathcal H,i+1)=N\leq (m-i)p(\mathcal H,i)$, as desired.
\end{proof}

We are now ready to prove the main result of this paper.

\begin{proof}[\proofname{ of \bf Theorem~\ref{CsikvariConjHgraph}.}]
Let $k\ge 2$ and let $\mathcal{H}$ be a $k$-graph with $m$ edges.
Write $\nu:=\nu(\mathcal{H})$ and $p_i:=p(\mathcal{H},i)$.
Denote by $\theta(\mathcal{H})$ the largest root of $\phi(\mathcal H,x)$.
Then $\theta(\mathcal{H})=\lambda(\mathcal{H})^k$ as claimed before.
We next prove the corresponding bounds for $\theta(\mathcal{H})$ by considering the following three cases.

\medskip

\noindent
{\bf Case 1: $\nu=1$.} In this case, $\mathcal H$ is intersecting, $\phi(\mathcal{H},x)=x-m$, and $\theta(\mathcal{H})=m.$

\medskip

\noindent
{\bf Case 2}: $\nu=2$.
In this case,  we have $p_2\ge 1$ and $\phi(\mathcal{H},x)=x^2-mx+p_2$.
Hence,
\[
\theta(\mathcal{H})=\frac{m+\sqrt{m^2-4p_2}}{2}
\leq \frac{m+\sqrt{m^2-4}}{2}=\frac{m+\sqrt{m^2-4(\nu-1)}}{2},
\]
and equality holds if and only if $p_2=1$.

\medskip

\noindent\textbf{Case 3: $\nu\ge 3$.}
Define
\[
\varrho(s):=\frac{m+\sqrt{m^2-4s}}{2},~ 0\leq s\leq\frac{m^2}{4}.
\]
Observe that the function $\varrho(s)$ is strictly decreasing on its domain.
Set $x_0:=\varrho(\nu-1)$.
As $m\geq \nu$ and $\nu\geq 3$, we have $\frac{m^2}{4}\geq \nu-1$, which implies that $x_0$ is well-defined.
We next prove that
$\theta(\mathcal{H})=\lambda(\mathcal{H})^k< x_0$.
Since $\theta(\mathcal{H})$ is the largest real root of $\phi(\mathcal{H},x)$,
it suffice to prove that
$\phi(\mathcal{H},x)>0$ for every $x\geq x_0$.

\begin{claim}\label{claim111}
$x_0\geq m-1$ and $x_0^2-mx_0=1-\nu$.
\end{claim}

\begin{proof}
Since $\nu\le m$,
we have $m^2-4(\nu-1)\ge m^2-4(m-1)=(m-2)^2$, which implies that
$$
x_0=\frac{m+\sqrt{m^2-4(\nu-1)}}{2}\geq\frac{m+(m-2)}{2}=m-1.
$$
Moreover, $x_0^2 - m x_0 = 1 - \nu$ is obtained by direct calculation.
\end{proof}

\begin{claim}\label{claim222}
If $\nu\geq 4$ and $x\geq x_0$,
then
$$\sum_{i=4}^{\nu}(-1)^i p_i\,x^{\nu-i}>0.$$
\end{claim}

\begin{proof}
If $m=4$, then $\nu=4$
and $\sum_{i=4}^{\nu}(-1)^i p_i\,x^{\nu-i}=p_4=1>0$.
Assume that $m\geq 5$.
By Lemma~\ref{lem:rec} and Claim~\ref{claim111},
for $4\leq i< \nu$ and $x\geq x_0$,
one may calculate that
$$
\frac{p_i x^{\nu-i}}{p_{i+1}x^{\nu-i-1}}
=\frac{p_ix}{p_{i+1}}  \geq  \frac{(i+1)x_0}{m-i}
\geq \frac{(i+1)(m-1)}{m-4} \geq i+1>4.
$$
This implies that for every $x\geq x_0$, we have
\begin{equation}\label{equation1111}
p_4 x^{\nu-4}>p_5 x^{\nu-5}>\cdots>p_\nu x^{\nu-\nu}>0.
\end{equation}
Note that $\sum_{i=4}^{\nu}(-1)^i p_i\,x^{\nu-i}$ is alternating and its first term is positive.
The claim follows from~\eqref{equation1111}.
\end{proof}

Observe that
$$\phi(\mathcal{H},x)-\sum_{i=4}^{\nu}(-1)^i p_i x^{\nu-i}
=x^{\nu-3}(x A(x)-p_3),$$
where  $A(x)=x^2-mx+p_2$.
Combining this and Claim~\ref{claim222}, to prove that $\phi(\mathcal{H},x)>0$ for every $x\geq x_0$,
it suffices to prove
\begin{equation}\label{equation2222}
x A(x)-p_3 >0 \text{ for every } x\geq x_0.
\end{equation}
Using Claim~\ref{claim111}, we get that $A'(x)=2x-m\geq 2x_0-m\geq m-2>0$ for every $x\geq x_0$.
Hence,  for every $x\geq x_0$, we have  
\begin{equation}\label{equation3333}
A(x)\geq A(x_0)=x_0^2-mx_0+p_2=p_2- \nu+1,
\end{equation}
where the last equality follows from Claim~\ref{claim111}.
Note that every two distinct edges in a maximum matching form a $2$-matching, so we get $p_2\geq \binom{\nu}{2}$, which implies that $p_2-\nu+1>0$.
Applying Lemma~\ref{lem:rec} with $i=2$, we get $p_3\leq \frac{m-2}{3}p_2$.
Combining these facts, for every $x\geq x_0$, we deduce that
\begin{align*}
x A(x)-p_3&\geq x_0(p_2-(\nu-1))-\frac{m-2}{3}p_2 && \text {(by~\eqref{equation3333})}\\
&\geq (m-1) (p_2-(\nu-1))-\frac{m-2}{3}p_2  && \text {(by Claim~\ref{claim111})}\\
&=\frac{2m-1}{3}p_2-(m-1)(\nu-1)  \\
&\geq \frac{2m-1}{3}\binom{\nu}{2} -(m-1)(\nu-1)&& \text {(as $p_2\geq \binom{\nu}{2}$)}\\
&=\frac{\nu-1}{6} (\nu(2m-1)- 6(m-1))\\
&\geq \frac{\nu-1}{6} (3(2m-1)- 6m+6), && \text {(as~$\nu\geq 3$)}\\
&>0
\end{align*}
which implies~\eqref{equation2222}.
Hence, $\theta(\mathcal{H})=\lambda(\mathcal{H})^k< x_0$, as desired.  $\hfill\qed$

For $\nu\geq 2$, by Cases 2 and 3, we get that
$$\lambda(\mathcal{H})^k=\theta(\mathcal{H})
\leq \frac{m+\sqrt{m^2-4(\nu-1)}}{2},$$
with equality if and only if $\nu=2$ and $p_2=1$.
Combining this fact and Case 1, for $\nu\geq 1$, we further obtain that
$$\lambda(\mathcal{H})^k=\theta(\mathcal{H}) \leq m,$$
with equality if and only if $\nu=1$.
The proof is completed.
\end{proof}

\section{Concluding remarks}

In this paper, we proved that the largest matching root of a $k$-graph $\mathcal{H}$ with $m$ edges is at most $m^{1/k}$,
and equality holds if and only if $\mathcal{H}$ is intersecting.
For $k=2$, after deleting all isolated vertices, the resulting graph must be the star $K_{1,m}$ or a triangle, thereby confirming Csikv\'ari's conjecture.
Moreover, for $\nu(\mathcal{H})\ge 2$,  Theorem~\ref{CsikvariConjHgraph} provides a tight bound for $\lambda(\mathcal{H})$ and characterizes the parameters of extremal $k$-graphs.

Building on Theorem~\ref{CsikvariConjHgraph},
a natural problem is to determine the following parameter:
$$f(m,k,\nu):=\max\{\lambda(\mathcal{H}):
\mathcal{H} \text{ is a $k$-graph with $m$ edges and } \nu(\mathcal{H})=\nu\}.$$
Clearly, Theorem~\ref{CsikvariConjHgraph} determined $f(m,k,1)$ for $k\geq 2$ and $f(m,k,2)$ for $k\geq 3$.
For every fixed $k\geq 2$ and $\nu\geq 3$, the next result gives an asymptotic estimate of $f(m,k,\nu)$ provided that $m$ is large.




\begin{proposition}
For every fixed $k\geq 2$ and $\nu\geq 3$,
$$f(m,k,\nu)= (1-o_m(1))m^{1\over k},$$
where $o(1)_m$ is a quantity that tends to zero as $m\to \infty$.
\end{proposition}
\begin{proof}
Denote by $\mathcal{S}_m$ the linear star with $m$ edges, that is, the $k$-graph whose vertex set admits a partition
$\{v\}\cup V_1 \cup\cdots  \cup V_m $
with $|V_1| =\cdots =|V_m| =k-1$
and the edge set is $\{\{v\}\cup V_i : i = 1,\ldots,m\}$.
As $\mathcal{S}_m$ is intersecting, we get
$\lambda(\mathcal{S}_m)=m^{1/k}$ by Theorem~\ref{CsikvariConjHgraph}.
Write $\mathcal H_{m,k,\nu}$ for the $k$-graph obtained from
$\mathcal{S}_{m-\nu+1}$ by adding $\nu-1$ new pairwise disjoint edges that are disjoint from $V(S_{m-\nu+1})$.
By Lemma~\ref{disjointunion}, we deduce that
$$f(m,k,\nu)\geq \lambda(\mathcal H_{m,k,\nu})
=\lambda(\mathcal{S}_{m-\nu+1})
=(m-\nu+1)^{1/k}=(1-o_m(1))m^{1/k}.$$
On the other hand, we have $f(m,k,\nu)\leq m ^{1/k}$ by Theorem~\ref{CsikvariConjHgraph}.
The statement follows.
\end{proof}

\end{document}